\documentclass[letterpaper, 10 pt, conference]{ieeeconf}  

\IEEEoverridecommandlockouts                              

\usepackage{cite}
\usepackage{amsfonts}
\usepackage{algorithmic}
\usepackage{graphicx}
\usepackage{textcomp}
\usepackage{graphics} 
\usepackage{amsmath} 
\usepackage{amssymb}  
\usepackage{euscript}
\usepackage{enumerate}

\newtheorem{defi}{Definition}[section]

\newtheorem{lem}[defi]{Lemma}
\newtheorem{prop}[defi]{Proposition}
\newtheorem{thm}[defi]{Theorem}
\newtheorem{rem}[defi]{Remark}
\newtheorem{cor}[defi]{Corollary}

\title{\LARGE \bf
Risk-Aware Stability, Ultimate Boundedness, and Positive Invariance
}

\author{Masako Kishida
\thanks{M. Kishida is with the National Institute of Informatics, Tokyo, Japan 
        {\tt\small kishida@nii.ac.jp}}%
}

\begin{document}

\maketitle
\thispagestyle{empty}
\pagestyle{empty}

\begin{abstract}
This paper introduces the notions of stability, ultimate boundedness, and positive invariance for stochastic systems in the view of risk.
More specifically, those notions are defined in terms of the worst-case Conditional Value-at-Risk (CVaR),
which quantifies the worst-case conditional expectation of losses exceeding a certain threshold over a set of possible uncertainties. 
Those notions allow us to focus our attention on the tail behavior of stochastic systems in the analysis of dynamical systems and the design of controllers.
Furthermore, some event-triggered control strategies that guarantee ultimate boundedness and positive invariance with specified bounds are derived using the obtained results and illustrated using numerical examples.

\end{abstract}

\section{Introduction}\label{sec:intro}
Whether it is a financial portfolio or engineering system, we often need to accept a certain level of risk to balance the pros and cons of spending costs in the decision-making processes under uncertainties.
This is particularly true when the uncertainty distributions have an unbounded support because a guarantee of 100\% ideal satisfaction is impossible or requires an infinite amount of cost.
Therefore, risk has been studied for a long time, not only in the financial industry \cite{Mar68, KleB76, MarT00, ZhuF09}, but also in the wide areas of engineering \cite{Cor68, Whi81, Whi84} using many different approaches.

The notions of stability, ultimate boundedness, and positive invariance are fundamental in the analysis of dynamical systems and in the design of controllers \cite{Bla99}. 
To deal with stochastic uncertainties in dynamical systems, the concept of probabilistic stability was introduced in \cite{Kus65}.
Later, the concepts of probabilistic set invariance and ultimate boundedness were introduced for discrete-time linear systems in \cite{KofDS12} and extended to continuous-time linear systems in \cite{KofDS16}. Those probabilistic notions are defined using chance constraints.
Another popular approach to dealing with stochastic uncertainties is to consider mean square and $p$-stability \cite{KanL01}.
In this direction, ultimate boundedness \cite{Zak67} and positive invariance \cite{BenBD02,BouB02} have been also investigated.
However, those notions are not applicable to guarantee that the expected value of the constraint-violating cases is small.

To take into account this additional consideration of risk, 
this paper introduces the definitions of stability, ultimate boundedness, and positive invariance in terms of the worst-case Conditional Value-at-Risk (CVaR) for stochastic systems. 
CVaR is a relatively new risk measure that is defined as the conditional expectation of losses exceeding a certain threshold \cite{RocU00}. 
The worst-case CVaR is the supremum of CVaR over a set of possible disturbances \cite{ZhuF09, HuaZFF08}. 
Using the (worst-case) CVaR, the tail behavior of the stochastic systems can be quantified; it quantifies the risk which has a low probability of occurring, but if it does occur, it will result in a large loss.
Assessing the tail risk is especially beneficial when the uncertainty distribution has a fat tail. In such a case, the use of the chance constraints or standard expected value for analysis may result in a significant loss.
Because the (worst-case) CVaR is a coherent risk measure \cite{ZhuF09}, it enjoys nice mathematical properties.
In particular, we see the worst-case CVaR on the squared norm of the states using the first two moments of the disturbance allows 
us to obtain elegant results for discrete-time linear stochastic systems. 
Such a problem setup can be easily applied to cases wherein the disturbance is non-Gaussian or the probability density function is not available.

The rest of the paper is organized as follows. 
After introducing notation, definitions, and properties of the worst-case CVaR as well as some basic results in
Section \ref{sec:pre}, 
Section \ref{sec:stab} and Section \ref{sec:rob} present the definitions of stability, ultimate boundedness, and positive invariance using the worst-case CVaR,
without inputs and with bounded inputs, respectively.
Based on the results in Section \ref{sec:rob}, approaches to risk-aware event-triggered control are developed in Section \ref{sec:event}, which is followed by conclusion in Section  \ref{sec:conc}.
\section{Preliminaries}\label{sec:pre}
\subsection{Notation}
The sets of real numbers, real vectors of length $n$, and real matrices
of size $n \times m$ are denoted by $\mathbb{R}$, $\mathbb{R}^n$,  and $\mathbb{R}^{n\times m}$, respectively. 
The sets of nonnegative numbers, nonnegative integers, and positive integers are denoted by $\mathbb{R}_{\geq0}$, $\mathbb{Z}_{\geq 0}$ and $\mathbb{Z}_{> 0}$, respectively.
For $M\in \mathbb{R}^{n\times n}$, $M \succ 0$ indicates $M$ is positive definite. 
$M^\top$ denotes the transpose of a real matrix $M$ and $\text{Tr}(M)$ denotes the trace of $M$.
$I_n$ denotes the identity matrix of size $n$. 
The Kronecker product of two matrices $X$ and $Y$ is denoted as $X\otimes Y$.
For $x\in\mathbb{R}$, $x^+ = \max\{x, 0\}$. 
For a vector $v\in \mathbb{R}^n$, $\|v\|$ denotes the Euclidean norm. 
For a matrix $M$, $\|M\|$ denotes the maximum singular value norm. 
Recall that a function $\gamma: \mathbb{R}_{\geq0} \rightarrow  \mathbb{R}_{\geq0} $ is a $\mathcal{K}$ function if it is continuous, strictly increasing and $\gamma(0)=0$.
A function $\beta: \mathbb{R}_{\geq0} \times  \mathbb{R}_{\geq0} \rightarrow  \mathbb{R}_{\geq0} $ is a $\mathcal{KL}$ function if for each fixed $t \geq 0$, $\beta(s,t)\in\mathcal{K}$  with respect to $s$  and for each fixed $s\geq 0$, $\beta(s,t)$  is decreasing with respect to $t$ and $ \lim_{t\rightarrow\infty} \beta(s,t)=0$.
\subsection{Conditional Value-at-Risk}
Let $\mu \in\mathbb{R}^n$ be the mean and $\Sigma \in\mathbb{R}^{n\times n}$ be the covariance matrix of the random vector  $\xi \in\mathbb{R}^n$ under the true distribution $\mathbb{P}$, which is the probability law of $\xi$. 
Thus, it is implicitly assumed that the random vector $\xi$ has finite second-order moments.
Let $\mathcal{P}$ denote the set of all probability distributions on $\mathbb{ R}^n$ that have the same first- and second-order moments as $\mathbb{P}$, i.e.,
\begin{align*} 
 \mathcal{P}= \left\{\mathbb{P} : \mathbb{E}_{\mathbb{P}}\left[ \begin{bmatrix}\xi_i\\ 1\end{bmatrix}\begin{bmatrix}\xi_j\\ 1\end{bmatrix}^{\!\top} \right]
=  \begin{bmatrix}  \Sigma  \delta_{ij}  & 0 \\ 0^\top & 1 \end{bmatrix}, \forall i, j\right\}.
\end{align*}
Here $\delta_{ij}$ denotes the Kronecker delta and $\mathbb{E}_{\mathbb{P}}[ \cdot]$ denotes the expectation with respect to $\mathbb{P}$.
The true underlying probability measure $\mathbb{P}$ is not known exactly, but it is known that  $\mathbb{P}\in \mathcal{P}$. 

\begin{defi}[Conditional Value-at-Risk \cite{RocU00,ZymKR13-b}]
For a given measurable loss function $L : \mathbb{R}^n \rightarrow  \mathbb{R}$, a probability distribution $\mathbb{P}$ on $\mathbb{R}^n$ and a level 
 $\varepsilon \in (0, 1)$, the CVaR at $\varepsilon$ with respect to $\mathbb{P}$ is defined as
 \begin{align*}
\mathbb{P}\text{-CVaR}_{\varepsilon}[L({\xi})] =\inf_{\beta \in \mathbb{R}} \left\{ \beta + \frac{1}{\varepsilon}\mathbb{E}_{\mathbb{P}}[(L({\xi})-\beta)^+]\right\}.
\end{align*}
\end{defi}
CVaR is the conditional expectation of loss above the ($1-\varepsilon$)-quantile of the loss function \cite{ZymKR13-b} and quantifies the tail risk (see Figure \ref{fig:cvar}).

The worst-case CVaR  is the supremum of CVaR over a given set of probability distributions as defined below: 
\begin{defi}[Worst-case CVaR \cite{ZymKR13-b}]
The worst-case CVaR over $\mathcal{P}$ is given by
 \begin{align*}
\sup_{\mathbb{P}\in \mathcal{P}}\mathbb{P}\text{-CVaR}_{\varepsilon}[L({\xi})]
=\inf_{\beta \in \mathbb{R}} \left\{ \beta + \frac{1}{\varepsilon}\sup_{\mathbb{P}\in \mathcal{P}}\mathbb{E}_{\mathbb{P}}[(L({\xi})-\beta)^+]\right\}.
\end{align*}
\end{defi}
Here, the exchange between the supremum and infimum is justified by the stochastic saddle point theorem \cite{ShaK02}.

\begin{figure}[b]\centering
 \includegraphics[width=.98\linewidth, viewport =10 10 570 280, clip]{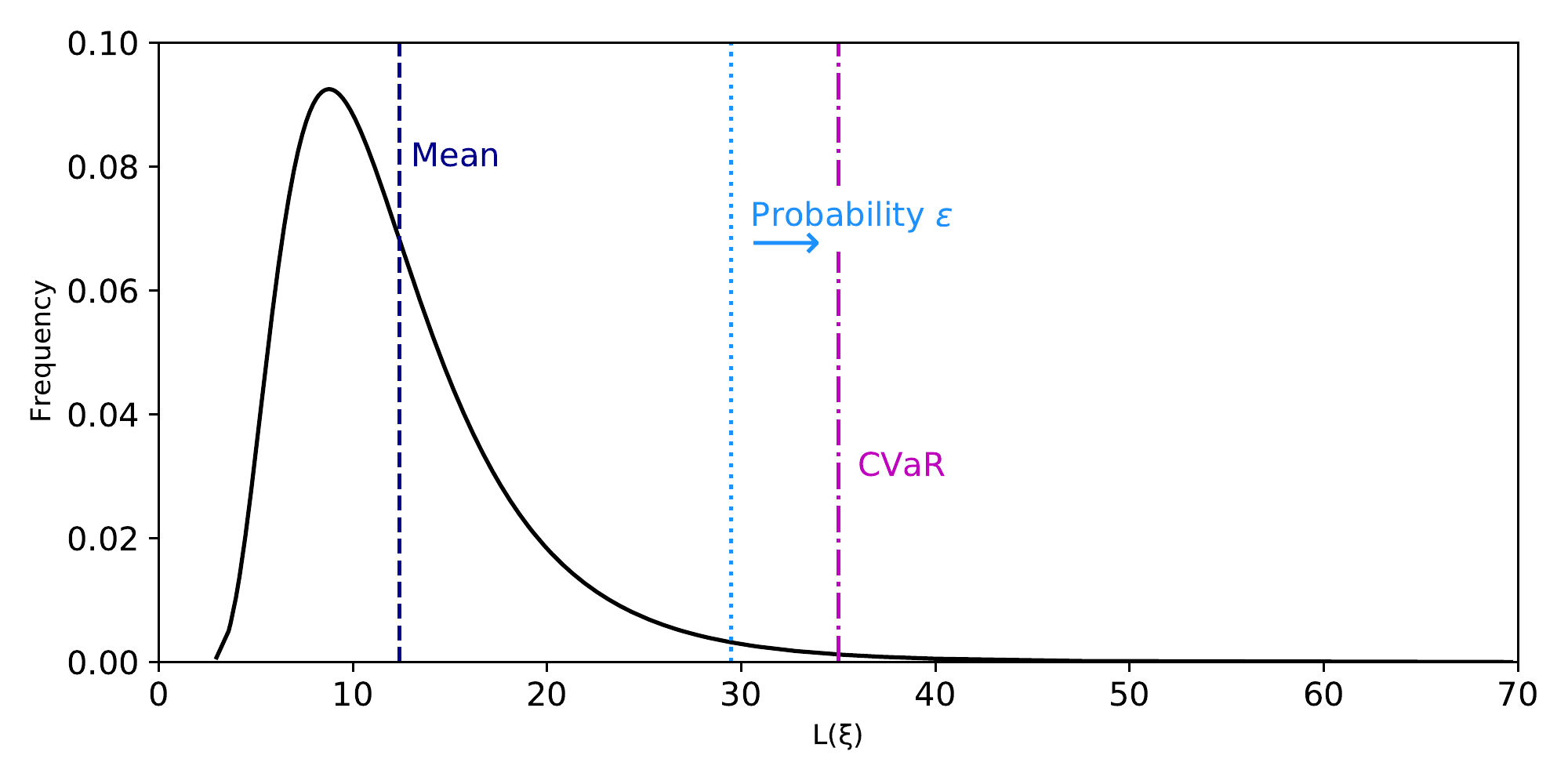}
\caption{Illustration of mean and CVaR} 
\label{fig:cvar}
\end{figure}

If $L(\xi)$ is quadratic with respect to $\xi$, the worst-case CVaR can be computed by a semidefinite program \cite{ZymKR13-b}, \cite{ZymKR13}.
Furthermore, if the mean of the random vector  $\xi$ is zero, the following easy-to-compute bounds are obtained. 
\begin{lem}[Bounds for worst-case CVaR \cite{KisC23}]\label{lem:CVaR_zero_mean_bound}
Suppose $\mu=0$ and  $L(\xi)= \| A\xi + b \|^2+c$ with some $A \in \mathbb{R}^{m \times n}$, $b\in \mathbb{R}^n$ and $c\in \mathbb{R}$, then
 \begin{align*}
c  + b^\top b + \frac{1}{\varepsilon}\left(\text{Tr}(\Sigma A^\top A)\right) \leq
 &\ \sup_{\mathbb{P}\in \mathcal{P}}\mathbb{P}\text{-CVaR}_{\varepsilon}[L({\xi})] \\
 \leq &\  c + \frac{1}{\varepsilon}\left(\text{Tr}(\Sigma A^\top A)+b^\top b\right).
 \end{align*}
 If $b=0$ and $c=0$, then it follows that
  \begin{align*}
\sup_{\mathbb{P}\in \mathcal{P}}\mathbb{P}\text{-CVaR}_{\varepsilon}[L({\xi})] 
= \frac{1}{\varepsilon}\text{Tr}(\Sigma A^\top A).
 \end{align*}
\end{lem}

To deal with dynamical systems, the following is a simple, but useful result.
\begin{lem}\label{lem:aug}
Suppose $\xi_1, \xi_2, ..., \xi_m \in \mathbb{R}^n$ are independent and identically distributed random vectors under the true distribution $\mathbb{P}\in \mathcal{P}$.
Let $\bar{\xi} = [\xi_1^\top, \xi_2^\top, ..., \xi_m^\top]^\top$.
Then, $\bar{\xi} \in \mathbb{R}^{nm}$ is a random vector with the mean zero and covariance $I_m\otimes \Sigma$. 
Thus, the true underlying probability measure $\mathbb{P}$ for $\xi$ satisfies $\mathbb{P}\in \mathcal{P}_{\text{aug}}$, where
\begin{align*}
 \mathcal{P}_{\text{aug}}=  \left\{\mathbb{P}\! :\! \mathbb{E}_{\mathbb{P}}\left[ \begin{bmatrix}\bar{\xi}_{i}\\ 1\end{bmatrix}\begin{bmatrix}\bar{\xi}_j\\ 1\end{bmatrix}^{\!\top} \right]
= \begin{bmatrix} \! (I_m\otimes \Sigma) \delta_{ij} & 0\\ 0 & 1 \end{bmatrix}, \forall i, j\right\}.
\end{align*}
Moreover, for a matrix $A \in \mathbb{R}^{\ell \times nm}$, it holds that
  \begin{align*}
\sup_{\mathbb{P}\in \mathcal{P}_{\text{aug}}}\mathbb{P}\text{-CVaR}_{\varepsilon}[\|A\xi\|^2] 
= \frac{1}{\varepsilon}\text{Tr}((I_m\otimes \Sigma)A^\top A).
 \end{align*}
\end{lem}
\begin{proof}
Follows from the assumption that $\xi_1, \xi_2, ..., \xi_m \in \mathbb{R}^n$ are independent and identically distributed.
The second part follows from Lemma \ref{lem:CVaR_zero_mean_bound}.
\end{proof}
One reason that the (worst-case) CVaR is popular for risk assessment is its mathematically attractive properties of coherency.
\begin{prop}[Coherence properties \cite{ZhuF09,Art99}]\label{prop:pro}
The worst-case CVaR is a coherent risk measure, i.e., it satisfies the following properties:
Let $L_1 = L_1(\xi)$ and $L_2 = L_2(\xi)$ be two measurable loss functions.
\begin{itemize}
\item Sub-additivity: For all $L_1$ and $L_2$,
\begin{align*}
&\sup_{\mathbb{P}\in \mathcal{P}}\mathbb{P}\text{-CVaR}_{\varepsilon}[L_1+L_2]\\
 & \leq \sup_{\mathbb{P}\in \mathcal{P}}\mathbb{P}\text{-CVaR}_{\varepsilon}[L_1] +
\sup_{\mathbb{P}\in \mathcal{P}}\mathbb{P}\text{-CVaR}_{\varepsilon}[L_2]; 
\end{align*}
\item Positive homogeneity: For a positive constant $a>0$, 
\begin{align*}
\sup_{\mathbb{P}\in \mathcal{P}}\mathbb{P}\text{-CVaR}_{\varepsilon}[aL_1] =a\sup_{\mathbb{P}\in \mathcal{P}}\mathbb{P}\text{-CVaR}_{\varepsilon}[L_1]; 
\end{align*}
\item Monotonicity: If $L_1\leq L_2$ almost surely, 
\begin{align*}
\sup_{\mathbb{P}\in \mathcal{P}}\mathbb{P}\text{-CVaR}_{\varepsilon}[L_1] \leq
\sup_{\mathbb{P}\in \mathcal{P}}\mathbb{P}\text{-CVaR}_{\varepsilon}[L_2];
\end{align*}
\item Translation invariance: For a constant $c$,
\begin{align*}
\sup_{\mathbb{P}\in \mathcal{P}}\mathbb{P}\text{-CVaR}_{\varepsilon}[L_1+c] = \sup_{\mathbb{P}\in \mathcal{P}}\mathbb{P}\text{-CVaR}_{\varepsilon}[L_1] +
c.
\end{align*}
\end{itemize}
\end{prop}
\subsection{Useful Result}
This paper repeatedly utilizes the following inequality.
\begin{lem}[Cauchy-Schwarz inequality]\label{lem:ineq}
For two vectors $x, y \in \mathbb{R}^n$, it holds that
\begin{align*}
2x^\top y \leq \alpha^2 \|x\|^2 +\frac{1}{\alpha^2} \|y\|^2, \ \forall \alpha > 0.
\end{align*}
\end{lem} 


\section{Linear System without Inputs} \label{sec:stab}
This section considers linear systems subject to stochastic disturbances. 
After introducing a system model, stability, ultimate boundedness, and positive invariance are investigated.
\subsection{System Model}
Consider the discrete-time linear stochastic system 
\begin{align}
x_{t+1} =Ax_t+Ew_t, \label{eq:no_control}
\end{align}
where $x_t \in \mathbb{R}^{n_x}$ is the state and $w_t \in \mathbb{R}^{n_w}$ is the disturbance, respectively, at discrete time instant $t \in \mathbb{Z}_{\geq 0}$. 
$A \in \mathbb{R}^{n_x\times n_x}$ and $E \in \mathbb{R}^{n_x\times n_w}$ are constant matrices.  
It is assumsed that the initial condition $x_0 \in  \mathbb{R}^{n_x}$ is given, and that $w_t$ are independent and identically distributed random vectors with the mean zero and covariance $\Sigma_w \succ 0$ for all $t \in \mathbb{Z}_{\geq 0}$.
The true underlying probability measure $\mathbb{P}$ is not known exactly, but it is known that  $\mathbb{P}\in \mathcal{P}$, where
\begin{align} \label{eq:P1}
 \mathcal{P}\!= \! \left\{\mathbb{P}\! :\! \mathbb{E}_{\mathbb{P}}\left[ \begin{bmatrix}w_i\\ 1\end{bmatrix} \begin{bmatrix}w_j\\ 1\end{bmatrix}^{\!\top} \right]
\!\!=\begin{bmatrix} \Sigma_w \delta_{ij}  & 0 \!\\\! 0^\top & 1 \end{bmatrix}, \forall i, j\right\}.
\end{align}

For $t\in \mathbb{Z}_{> 0}$, the state evolution of \eqref{eq:no_control} can be expressed as
\begin{align}
x_t = F_t x_0 +  H_t \bar{w}_t, 
\end{align}
using
\begin{align}\begin{aligned}\label{eq:abb}
\bar{w}_t &= [w_0^\top, \ w_1^\top, \cdots, w_{t-1}^\top]^\top,\\
F_t &= A^t, \\
H_t &=\begin{bmatrix}A^{t-1}E&A^{t-2}E&\cdots & E\end{bmatrix}.
\end{aligned}\end{align}
With this notation,  from Lemma \ref{lem:aug}, $\bar{w}_t\in \mathbb{R}^{n_wt}$ is a random vector with the mean zero and covariance $I_t\otimes \Sigma_w$. 
Thus, the true underlying probability measure $\mathbb{P}$ for $\bar{w}_t$ satisfies $\mathbb{P}\in \mathcal{P}_t$, where
\begin{align}
 \mathcal{P}_t \!= \! \left\{\mathbb{P}\! :\! \mathbb{E}_{\mathbb{P}}\left[ \begin{bmatrix}\bar{w}_{t,i}\\ 1\end{bmatrix}  \begin{bmatrix}\bar{w}_{t,j}\\ 1\end{bmatrix}^{\top} \right]
\!\!= \begin{bmatrix} (I_t\otimes \Sigma_w) \delta_{ij} & 0\!\\\! 0 & 1 \end{bmatrix}, \forall i, j\right\}.
\end{align}

\subsection{Practical Stability and Ultimate Boundedness}
Here, the stability notion in terms of the worst-case CVaR is introduced.
\begin{defi}[Practical stability] \label{defi:stability}
The system \eqref{eq:no_control} is practically asymptotically worst-case CVaR stable if there exist $\beta \in \mathcal{KL}$ and a constant $c \geq 0$ such that
\begin{align}\begin{aligned}
&\sup_{\mathbb{P}\in \mathcal{P}_t}\mathbb{P}\text{-CVaR}_{\varepsilon}[\|x_t\|^2]
 \leq  \beta(\|x_0\|^2,t) +c, \ \forall t \in \mathbb{Z}_{\geq 0}.
\end{aligned}\end{align}
\end{defi}

The following lemma provides a sufficient condition for a system to be practically asymptotically worst-case CVaR stable.
\begin{lem}[Sufficient condition for practical stability] \label{lem:stability}
If $\|A\| < 1$ and  $(A,E)$ is reachable, then the linear system \eqref{eq:no_control} 
is practically asymptotically worst-case CVaR  stable.
\end{lem}
\begin{proof}
Using Lemma \ref{lem:ineq}  and norm sub-multiplicativity, 
\begin{align}\begin{aligned}
\|x_t\|^2 &= \|F_t x_0  + H_t \bar{w}_t\|^2\\
& \leq   (1+\alpha^2) \|F_t\|^2 \|x_0\|^2 +\left(1+\frac{1}{\alpha^2}\right) \| H_t \bar{w}_t\|^2\\
& \leq   (1+\alpha^2)\lambda^{t} \|x_0\|^2+ \left(1+\frac{1}{\alpha^2}\right) \| H_t \bar{w}_t\|^2
\end{aligned}\end{align}
for any $\alpha>0$, where $\lambda = \|A\|^2 \in [0,1)$.
Using Proposition \ref{prop:pro} along with Lemma \ref{lem:aug}, it follows that
\begin{align}\begin{aligned}
& \sup_{\mathbb{P}\in \mathcal{P}_t}\mathbb{P}\text{-CVaR}_{\varepsilon}[\|x_t\|^2]\\
& \leq (1+\alpha^2)\lambda^{t}\|x_0 \|^2  +\left(1+\frac{1}{\alpha^2}\right)\frac{1}{\varepsilon}\text{Tr}( (I_t\otimes \Sigma_w)H_t^\top H_t)\\
& \leq (1+\alpha^2)\lambda^{t} \|x_0 \|^2 + \left(1+\frac{1}{\alpha^2}\right) \frac{1}{\varepsilon}\text{Tr}(P).
\end{aligned}\end{align}
Here, we used
\begin{align}\begin{aligned}
 \text{Tr}( (I_t\otimes \Sigma_w)H_t^\top H_t) 
 &= \text{Tr}\left( \sum_{k=0}^{t-1}(A^k) E\Sigma_w E^\top (A^\top)^k\right)\\
&\leq \text{Tr}\left(   \sum_{k=0}^{\infty}(A^k) E\Sigma_w E^\top (A^\top)^k\right)\\
&= \text{Tr}(P),\label{eq:ps}
\end{aligned}\end{align}
where $P\succ 0$ is the solution to the Lyapunov equation
\begin{align}
A PA^\top-P+E\Sigma_w E^\top = 0. \label{eq:lyap}
\end{align}
The last equality in \eqref{eq:ps} is due to $A$ is Schur and $(A, E)$ is reachable.
Thus, choosing 
\begin{align}
\beta(s,t) =(1+\alpha^2)s\lambda^t, \ c =\left(1+\frac{1}{\alpha^2}\right) \frac{1}{\varepsilon}\text{Tr}(P)
\end{align}
satisfies Definition \ref{defi:stability}.
\end{proof}

\begin{rem}
For linear systems, assuming that $A$ is Schur is sufficient for practically asymptotically worst-case CVaR stability as for other kinds of stability.
However, we impose the assumption $\|A\|< 1$ in this paper, and Lemma \ref{lem:stability} is focused on such a case.
\end{rem}

Here, we introduce a notion of worst-case CVaR ultimate bound, which is an extension of the probabilistic ultimate bound \cite{KofDS12}. 
The notion of ultimate bound is closely related to practical asymptotic stability.
\begin{defi}[Ultimate bound]
Let $f: \mathbb{R}^{n_x}\rightarrow \mathbb{R}$.
A domain $\mathcal{D}=\{x \in \mathbb{R}^{n_x}: f(x)\leq 0\}$ is a worst-case CVaR ultimate bound for the system \eqref{eq:no_control} if for every initial state $x_0$,  
there exists $T=T(x_0)>0$ such that $ \sup_{\mathbb{P}\in \mathcal{P}_t}\mathbb{P}\text{-CVaR}_{\varepsilon}[f(x_t)]\leq0$ for all $ t\geq T$.
\end{defi}

A worst-case CVaR ultimate bound can be found as below:
\begin{thm}[A ultimate bound]\label{thm:uub}
Let define the bounded set
\begin{align}
\mathcal{D}= \{x \in \mathbb{R}^{n_x}:\|x\|^2 - r^2 \leq 0\}
\end{align}
for $r>0$. Consider the system \eqref{eq:no_control}
and assume that $\|A\| < 1$ and  $(A,E)$ is reachable.
If
\begin{align}
r^2 > \frac{1}{\varepsilon}\text{Tr}(P), \label{eq:uub_cond}
\end{align}
then, $\mathcal{D}$ is a worst-case CVaR ultimate bound for \eqref{eq:no_control}.
\end{thm}
\begin{proof}
Under the condition \eqref{eq:uub_cond}, there exist $\delta, \alpha>0$ such that
\begin{align}
\delta = r^2- \left(1+\frac{1}{\alpha^2}\right)\frac{1}{\varepsilon}\text{Tr}(P)>0.
 \end{align}
Using the proof of Lemma \ref{lem:stability} with Proposition \ref{prop:pro},
\begin{align}\begin{aligned}
  &\sup_{\mathbb{P}\in \mathcal{P}_t}\mathbb{P}\text{-CVaR}_{\varepsilon}[\|x_t\|^2 - r^2]\\
 &\leq (1+\alpha^2)\lambda^{t} \|x_0 \|^2 +  \left(1+\frac{1}{\alpha^2}\right)\frac{1}{\varepsilon}\text{Tr}(P)- r^2\\
 &\leq (1+\alpha^2)\lambda^{t} \|x_0 \|^2 -\delta.
 \end{aligned} \end{align}
 Because the first term above approaches to $0$ as $t$ goes to infinity, there exists $T$ such that $ \sup_{\mathbb{P}\in \mathcal{P}_t}\mathbb{P}\text{-CVaR}_{\varepsilon}[\|x_t\|^2 - r^2]\leq0$ for all $ t\geq T$.
\end{proof}

\subsection{Positive Invariance}
The notion of positive invariance is yet another important concept, which is intimately related to ultimate boundedness.
\begin{defi}[Positively invariant set]
Let $f: \mathbb{R}^{n_x}\rightarrow \mathbb{R}$ be a continuous function. 
A domain $\mathcal{D}=\{x \in \mathbb{R}^{n_x}: f(x)\leq 0\}$ is a worst-case CVaR positively invariant set for the system \eqref{eq:no_control} if for any state $x_t \in \mathcal{D}$   
$\sup_{\mathbb{P}\in \mathcal{P}_k}\mathbb{P}\text{-CVaR}_{\varepsilon}[f(x_{t+k})]\leq0$ for all $ k\in \mathbb{Z}_{>0}$.
\end{defi}

A worst-case CVaR positively invariant set can be found as below:
\begin{thm}[A positively invariant set]\label{thm:inv}
Let define the bounded set
\begin{align}
\mathcal{D}= \{x \in \mathbb{R}^{n_x}:\|x\|^2 - r^2 \leq 0\},
\end{align}
for $r>0$.  Consider the system \eqref{eq:no_control} 
and assume  that $\|A\| < 1$ and  $(A,E)$ is reachable. If
\begin{align}
r^2 \geq \frac{1}{\varepsilon(1-\|A\|)^2 }\text{Tr}(\Sigma_wE^\top E),\label{eq:invariant_cond_alpha}
\end{align}
then, $\mathcal{D}$ is a worst-case CVaR positively invariant set for \eqref{eq:no_control}.
\end{thm}

\begin{proof}
First, we show that $x_t \in \mathcal{D}$ implies $\sup_{\mathbb{P}\in \mathcal{P}}\mathbb{P}\text{-CVaR}_{\varepsilon}[\|x_{t+1}\|^2- r^2]\leq0$.
Using Lemma \ref{lem:ineq}, for any $\alpha>0$,
\begin{align}\begin{aligned}
\|x_{t+1}\|^2 
&=\|Ax_t+Ew_t\|^2\\
&\leq (1+\alpha^2)\|Ax_t\|^2+\left(1+\frac{1}{\alpha^2}\right) \|Ew_t\|^2.
\end{aligned}\end{align}
Choose $\alpha>0$ that satisfies
\begin{align}
(1+\alpha^2)\|A\| =1\label{eq:alpha}.
\end{align}
Such an $\alpha$ always exists for $\|A\| < 1$.
Using Proposition \ref{prop:pro} as well as the norm sub-multiplicativity and Lemma \ref{lem:CVaR_zero_mean_bound}, it follows that
\begin{align}\begin{aligned}\label{eq:xt}
&\sup_{\mathbb{P}\in \mathcal{P}}\mathbb{P}\text{-CVaR}_{\varepsilon}[\|x_{t+1}\|^2]
  \leq
 \|A\|\|x_t\|^2 +\frac{1}{1-\|A\|}\frac{1}{\varepsilon}\text{Tr}(\Sigma_wE^\top E).
\end{aligned}\end{align}
Hence, if $x_t \in\mathcal{D}$, using the condition \eqref{eq:invariant_cond_alpha},  it follows that
\begin{align}\begin{aligned}
&\sup_{\mathbb{P}\in \mathcal{P}}\mathbb{P}\text{-CVaR}_{\varepsilon}[\|x_{t+1}\|^2-r^2] \\
 &\leq  \|A\|r^2+\frac{1}{1-\|A\|}\left(1-\|A\|\right)^2r^2 - r^2=0.
\end{aligned}\end{align}

Next, we show $\sup_{\mathbb{P}\in \mathcal{P}_k}\mathbb{P}\text{-CVaR}_{\varepsilon}[\|x_{t+k}\|^2- r^2]\leq0$ for some $k>0$ implies $\sup_{\mathbb{P}\in \mathcal{P}_{k+1}}\mathbb{P}\text{-CVaR}_{\varepsilon}[\|x_{t+k+1}\|^2- r^2]\leq0$.
Again, using Lemma \ref{lem:ineq}, for any $\alpha>0$,
\begin{align}\begin{aligned}
\|x_{t+k+1}\|^2 
&=\|Ax_{t+k}+Ew_t\|^2\\
&\leq (1+\alpha^2)\|Ax_{t+k}\|^2+\left(1+\frac{1}{\alpha^2}\right) \|Ew_{t+k}\|^2.
\end{aligned}\end{align}
Choose $\alpha>0$ that satisfies \eqref{eq:alpha}.
As before, we have
\begin{align}\begin{aligned}
&\sup_{\mathbb{P}\in \mathcal{P}}\mathbb{P}\text{-CVaR}_{\varepsilon}[\|x_{t+k+1}\|^2-r^2]\\
&  \leq
 \|A\|\sup_{\mathbb{P}\in \mathcal{P}}\mathbb{P}\text{-CVaR}_{\varepsilon}[\|x_{t+k}\|^2] +\frac{1}{1-\|A\|}\frac{1}{\varepsilon}\text{Tr}(\Sigma_wE^\top E)-r^2\\
 &\leq  \|A\|r^2+\frac{1}{1-\|A\|}\left(1-\|A\|\right)^2r^2 - r^2=0.
\end{aligned}\end{align}
This completes the proof.
\end{proof}

\begin{rem}
With $\alpha$ that does not satisfy \eqref{eq:alpha}, a similar argument still holds by replacing
the condition \eqref{eq:invariant_cond_alpha} with
\begin{align}
\frac{1}{\varepsilon}\text{Tr}(\Sigma_wE^\top E)\leq \alpha^2 \left(\frac{1}{1+\alpha^2}-\|A\|^2\right)r^2. \label{eq:invariant_cond}
\end{align}
This right-hand-side is maximized with $\alpha$ in \eqref{eq:alpha}.
\end{rem}

\section{Linear System with Bounded Inputs}\label{sec:rob}
In this section, linear systems with bounded inputs subject to stochastic disturbances are considered. 
To deal with such a case, the notion of input-to-state stability and robust positive invariance are developed using the worst-case CVaR.

\subsection{System Model}
Consider the discrete-time linear stochastic system 
\begin{align}
x_{t+1} =Ax_t+D v_t+Ew_t, \label{eq:sys}
\end{align}
where $x_t \in \mathbb{R}^{n_x}$ is the state, $v_t \in \mathcal{V}$ is the bounded input and $w_t \in \mathbb{R}^{n_w}$ is the disturbance, respectively, at discrete time instant $t \in \mathbb{Z}_{\geq 0}$. 
$A \in \mathbb{R}^{n_x\times n_x}$, $D \in \mathbb{R}^{n_x\times n_v}$ and $E \in \mathbb{R}^{n_x\times n_w}$ are constant matrices.  
It is assumsed that the initial condition $x_0 \in  \mathbb{R}^{n_x}$ is given, and that $w_t$ are independent and identically distributed random vectors defined for \eqref{eq:no_control}. 

We treat $v_t \in \mathcal{V}$ as a bounded disturbance and define
\begin{align}
 \mathcal{V} =\{v\in\mathbb{R}^{n_v}:\|v\| \leq d\} \label{eq:v}
\end{align}
for a given $d$.

For $t\in \mathbb{Z}_{> 0}$, the state evolution of \eqref{eq:sys} can be expressed by
\begin{align}
x_t = F_t x_0 + G_t\bar{v}_t + H_t \bar{w}_t, \label{eq:x_t}
\end{align}
using \eqref{eq:abb} and
\begin{align}\begin{aligned}
\bar{v}_t &= [v_0^\top, \ v_1^\top, \cdots, v_{t-1}^\top]^\top,\\
G_t &= \begin{bmatrix}A^{t-1}D&A^{t-2}D&\cdots & D\end{bmatrix}.
\end{aligned}\end{align}

\subsection{Practical Stability and Ultimate Boundedness}
As in the previous section, we define the input-to-state stability using the worst-case CVaR.
\begin{defi}[Practical stability] \label{defi:stability2}
The system \eqref{eq:sys} is practically worst-case CVaR input-to-state stable if there exist $\beta \in \mathcal{KL}$, $\gamma \in \mathcal{K}$, and $c \geq 0$ such that 
\begin{align}\begin{aligned}
&\sup_{\mathbb{P}\in \mathcal{P}_t}\mathbb{P}\text{-CVaR}_{\varepsilon}[\|x_t\|^2] \\
  &\leq   \beta( \|x_0\|^2, t) +\gamma (\max_{\tau \in [0,t-1]} (\|v_{\tau}\|^2)) +c, \ \forall t \in \mathbb{Z}_{\geq 0}.
\end{aligned}\end{align}
\end{defi}
Note that if a system is practically worst-case CVaR input-to-state stable, then it is practically asymptotically worst-case CVaR stable if there is no input.

The following lemma provides a sufficient condition for a system to be practically worst-case CVaR input-to-state stable.
\begin{lem}[Sufficient condition for practical stability]\label{lem:io_stability}
If $\|A\|<1$ and  $(A,E)$ is reachable, then the linear system \eqref{eq:no_control} 
is practically worst-case CVaR  input-to-state stable.
\end{lem}
\begin{proof}
From \eqref{eq:v}, $v_{\tau}$ is bounded. Write  $ \bar{v}_{t-1}= \max_{\tau \in [0,t-1]} (\|v_{\tau}\|))(\leq d)$.
Using Lemma \ref{lem:ineq} twice and norm sub-multiplicativity, 
\begin{align}\begin{aligned}
\|x_t\|^2 &= \|F_t x_0 + G_t\bar{v}_t   + H_t \bar{w}_t\|^2\\
& \leq   (1+\alpha_1^2) \lambda ^t\|x_0 \|^2 \\
&+  \left(1+\frac{1}{\alpha_1^2}\right) \left(1+\alpha_2^2\right) \left(\frac{ \|D\|}{1-\|A\|}\right)^2 \bar{v}_{t-1}^2\\
&+\left(1+\frac{1}{\alpha_1^2}\right) \left(1+\frac{1}{\alpha_2^2}\right)  \| H_t \bar{w}_t\|^2
\end{aligned}\end{align}
for any $\alpha_1, \alpha_2 >0$, where $\lambda =\|A\|^2 \in [0, 1)$. 
Here, we used
\begin{align}\begin{aligned}
\|G_t\bar{v}_t \|^2&  = \left\|\sum_{k=0}^{t-1}A^kDv_{t-1-k}\right\|^2 \\
&\leq \left\|\sum_{k=0}^{t-1}\|A\|^k\|D\|\|v_{t-1-k}\|\right\|^2 \\
&\leq \left\|\sum_{k=0}^{t-1}\|A\|^k\|D\| \max_{\tau \in [0,t-1]} (\|v_{\tau}\|)\right\|^2 \\
&\leq \left(\frac{ \|D\| }{1-\|A\|}\right)^2 \bar{v}_{t-1}^2.
\end{aligned}\end{align}
Using Proposition \ref{prop:pro} along with Lemma \ref{lem:aug}, it follows that:
\begin{align}\begin{aligned} \label{eq:sc_ps}
&\ \sup_{\mathbb{P}\in \mathcal{P}_t}\mathbb{P}\text{-CVaR}_{\varepsilon}[\|x_t\|^2]\\
&\ \leq   (1+\alpha_1^2)\lambda ^t\|x_0 \|^2 
+ \left(1+\frac{1}{\alpha_1^2}\right) \left(1+\alpha_2^2\right)\left(\frac{\|D\|}{1-\|A\|}\right)^2\bar{v}_{t-1}^2\\
&\ + \left(1+\frac{1}{\alpha_1^2}\right) \left(1+\frac{1}{\alpha_2^2}\right)  \frac{1}{\varepsilon}\text{Tr}(P),
\end{aligned}\end{align}
where $P\succ 0$ is the solution to the Lyapunov equation \eqref{eq:lyap}. 
Thus, choosing 
\begin{align}\begin{aligned}
&\beta(s,t) = (1+\alpha_1^2)s\lambda^t, \\ 
&\gamma(\bar{v}_{t-1}) =\left(1+\frac{1}{\alpha_1^2}\right) \left(1+\alpha_2^2\right)\left(\frac{\|D\|}{1-\|A\|}\right)^2 \bar{v}_{t-1}^2, \\ 
&c =\left(1+\frac{1}{\alpha_1^2}\right) \left(1+\frac{1}{\alpha_2^2}\right) \frac{1}{\varepsilon}\text{Tr}(P)
\end{aligned}\end{align}
satisfies Definition \ref{defi:stability2}.
\end{proof}

Similar to the previous section, we introduce a notion of worst-case CVaR ultimate bound with bounded input.
\begin{defi}[Ultimate bound]
Let $f: \mathbb{R}^{n_x}\rightarrow \mathbb{R}$. 
A domain $\mathcal{D}=\{x \in \mathbb{R}^{n_x}: f(x)\leq 0\}$ is a worst-case CVaR ultimate bound with bounded input for the system \eqref{eq:sys} if for every initial state $x_0$,  
there exists $T=T(x_0)>0$ such that $ \sup_{\mathbb{P}\in \mathcal{P}_t}\mathbb{P}\text{-CVaR}_{\varepsilon}[f(x_t)]\leq0$ for all $ t\geq T$.
\end{defi}

A worst-case CVaR ultimate bound can be found as below:
\begin{thm}[A ultimate bound]\label{thm:r_uub}
Let define the bounded set
\begin{align}
\mathcal{D}= \{x \in \mathbb{R}^{n_x}:\|x\|^2 - r^2 \leq 0\},
\end{align}
for $r>0$. Consider the system \eqref{eq:sys}
and assume  that $\|A\| < 1$ and  $(A,E)$ is reachable. If
\begin{align}
r^2 > \left( \frac{\|D\|}{1-\|A\|} d +\sqrt{ \frac{1}{\varepsilon}\text{Tr}(P)}\right)^2, \label{eq:r_uub_cond}
\end{align}
then, $\mathcal{D}$ is a worst-case CVaR ultimate bound with bounded input for \eqref{eq:sys}.
\end{thm}
\begin{proof}
Under the condition \eqref{eq:r_uub_cond}, there exist $\delta, \alpha_1>0$ such that 
\begin{align}\begin{aligned}
\delta &= r^2-
 \left(1+\frac{1}{\alpha_1^2}\right) \left( \frac{\|D\|}{1-\|A\|} d +\sqrt{ \frac{1}{\varepsilon}\text{Tr}(P)}\right)^2>0, \label{eq:delta}
 \end{aligned} \end{align}
By using $\bar{v}_{t-1}\leq d$ and substituting
 \begin{align}\begin{aligned}
 \alpha_2^2=\frac{1-\|A\|} {\|D\|d} \sqrt{\frac{1}{\varepsilon}\text{Tr}(P)} > 0
  \end{aligned} \end{align}
  in \eqref{eq:sc_ps}, it follows that
  \begin{align}\begin{aligned}
&\ \sup_{\mathbb{P}\in \mathcal{P}_t}\mathbb{P}\text{-CVaR}_{\varepsilon}[\|x_t\|^2]\\
&\ \leq   (1+\alpha_1^2)\lambda ^t\|x_0 \|^2 
+  \left(1+\frac{1}{\alpha_1^2}\right) \left( \frac{\|D\|}{1-\|A\|} d +\sqrt{ \frac{1}{\varepsilon}\text{Tr}(P)}\right)^2.
\end{aligned}\end{align}
Thus using Proposition \ref{prop:pro} and \eqref{eq:delta},
  \begin{align}\begin{aligned} 
&\ \sup_{\mathbb{P}\in \mathcal{P}_t}\mathbb{P}\text{-CVaR}_{\varepsilon}[\|x_t\|^2-r^2]
 \leq  (1+\alpha_1^2)\lambda ^t\|x_0 \|^2 -\delta.
\end{aligned}\end{align}
 Because the first term approaches to $0$ as $t$ goes to infinity, there exists $T$ such that $ \sup_{\mathbb{P}\in \mathcal{P}_t}\mathbb{P}\text{-CVaR}_{\varepsilon}[\|x_t\|^2 - r^2]\leq0$ for all $ t\geq T$.
\end{proof}

\subsection{Positive Invariance}
The notion of robust positively invariant set can be defined for system with bounded input.
\begin{defi}[Robust positively invariant set] 
Let $f: \mathbb{R}^{n_x}\rightarrow \mathbb{R}$ be a continuous function. 
A domain $\mathcal{D}=\{x \in \mathbb{R}^{n_x}: f(x)\leq 0\}$ is a worst-case CVaR robust positively invariant set for the system \eqref{eq:sys} if for any state $x_t \in \mathcal{D}$   
$\sup_{\mathbb{P}\in \mathcal{P}_k}\mathbb{P}\text{-CVaR}_{\varepsilon}[f(x_{t+k})]\leq0$ for all $ k\in \mathbb{Z}_{> 0}$.
\end{defi}

A robust positively invariant set can be found as follows.
\begin{thm}[A robust positively invariant set]\label{thm:r_inv}
Let define the bounded set
\begin{align}
\mathcal{D}= \{x \in \mathbb{R}^{n_x}:\|x\|^2 - r^2 \leq 0\},
\end{align}
for $r>0$. Consider the system \eqref{eq:sys}
and assume that $\|A\| < 1$ and  $(A,E)$ is reachable. If
\begin{align}
r^2 \geq \frac{1}{(1-\|A\|)^2 } \left(\|D\|d +\sqrt{\frac{1}{\varepsilon}\text{Tr}(\Sigma_wE^\top E)} \right)^2, \label{eq:invariant_cond_alpha2}
\end{align}
then, $\mathcal{D}$ is a worst-case CVaR robust positively invariant set for \eqref{eq:sys}. 
Note $d$ is defined in \eqref{eq:v}.
\end{thm}
\begin{proof}
First, we show that $x_t \in \mathcal{D}$ implies $\sup_{\mathbb{P}\in \mathcal{P}}\mathbb{P}\text{-CVaR}_{\varepsilon}[\|x_{t+1}\|^2- r^2]\leq0$.
Using Lemma \ref{lem:ineq}, for any $\alpha_1, \alpha_2>0$,
\begin{align}\begin{aligned}
\|x_{t+1}\|^2 
&=\|Ax_t+Dv_t+Ew_t\|^2\\
&\leq (1+\alpha_1^2)\|Ax_t\|^2+\left(1+\frac{1}{\alpha_1^2}\right)(1+\alpha_2^2) \|D\|^2 d^2\\
& \ +\left(1+\frac{1}{\alpha_1^2}\right)\left(1+\frac{1}{\alpha_2^2}\right)\|Ew_t\|^2. \label{eq:xt+1}
\end{aligned}\end{align}
Choose $\alpha_1>0$ that satisfies \eqref{eq:alpha} and $\alpha_2>0$ that satisfies
\begin{align}\begin{aligned}
\alpha_2^2 \|D\|d = \sqrt{\frac{1}{\varepsilon}\text{Tr}(\Sigma_wE^\top E)}.\label{eq:beta} 
\end{aligned}\end{align}
Using Proposition \ref{prop:pro} as well as the norm sub-multiplicativity and Lemma \ref{lem:CVaR_zero_mean_bound}, it follows that
\begin{align}\begin{aligned}\label{eq:wc_cvar_xt}
&\sup_{\mathbb{P}\in \mathcal{P}_{t+1}}\mathbb{P}\text{-CVaR}_{\varepsilon}[\|x_{t+1}\|^2]\\
  & \leq \|A\|\|x_t\|^2+\frac{1}{1-\|A\|}\left(\|D\|d +\sqrt{\frac{1}{\varepsilon}\text{Tr}(\Sigma_wE^\top E)} \right)^2.
\end{aligned}\end{align}
Hence, if $x_t \in\mathcal{D}$, using the condition \eqref{eq:invariant_cond_alpha2},  it follows that
\begin{align}\begin{aligned}
&\sup_{\mathbb{P}\in \mathcal{P}}\mathbb{P}\text{-CVaR}_{\varepsilon}[\|x_{t+1}\|^2-r^2] \\
 &\leq  \|A\|r^2+\frac{1}{1-\|A\|}\left(1-\|A\|\right)^2r^2 - r^2=0.
\end{aligned}\end{align}

Next, we show $\sup_{\mathbb{P}\in \mathcal{P}_k}\mathbb{P}\text{-CVaR}_{\varepsilon}[\|x_{t+k}\|^2- r^2]\leq0$ for some $k>0$ implies $\sup_{\mathbb{P}\in \mathcal{P}_{k+1}}\mathbb{P}\text{-CVaR}_{\varepsilon}[\|x_{t+k+1}\|^2- r^2]\leq0$.
Again, using Lemma \ref{lem:ineq}, for any $\alpha_1, \alpha_2>0$,
\begin{align}\begin{aligned}
\|x_{t+k+1}\|^2 
&=\|Ax_{t+k}+Dv_{t+k}+Ew_{t+k}\|^2\\
&\leq (1+\alpha_1^2)\|Ax_{t+k}\|^2+\left(1+\frac{1}{\alpha_1^2}\right)(1+\alpha_2^2) \|D\|^2 d^2\\
& \ +\left(1+\frac{1}{\alpha_1^2}\right)\left(1+\frac{1}{\alpha_2^2}\right)\|Ew_t\|^2. 
\end{aligned}\end{align}
Choose $\alpha_1>0$ that satisfies \eqref{eq:alpha} and $\alpha_2>0$ that satisfies \eqref{eq:beta}. 
As before, we have
\begin{align}\begin{aligned}
&\sup_{\mathbb{P}\in \mathcal{P}}\mathbb{P}\text{-CVaR}_{\varepsilon}[\|x_{t+k+1}\|^2-r^2]\\
&  \leq
 \|A\|\sup_{\mathbb{P}\in \mathcal{P}}\mathbb{P}\text{-CVaR}_{\varepsilon}[\|x_{t+k}\|^2] +\frac{1}{1-\|A\|}\frac{1}{\varepsilon}\text{Tr}(\Sigma_wE^\top E)-r^2\\
 &\leq  \|A\|r^2+\frac{1}{1-\|A\|}\left(1-\|A\|\right)^2r^2 - r^2=0.
\end{aligned}\end{align}
This completes the proof.
\end{proof}

\section{Event-triggered Control}\label{sec:event}
In this section, event-triggered control strategies are developed using the results in Section \ref{sec:rob}.

\subsection{System Model}
Consider the discrete-time linear control system subject to stochastic disturbance
\begin{align}
x_{t+1} =Ax_t+B u_t+Ew_t, \label{eq:sys2}
\end{align}
where $x_t \in \mathbb{R}^{n_x}$ is the state, $u_t \in \mathbb{R}^{n_u}$ is the control input  and $w_t \in \mathbb{R}^{n_w}$ is the disturbance, respectively, at discrete time instant $t \in \mathbb{Z}_{\geq 0}$. 
$A \in \mathbb{R}^{n_x\times n_x}$, $B \in \mathbb{R}^{n_x\times n_u}$ and $E \in \mathbb{R}^{n_x\times n_w}$ are constant matrices.
It is assumed that the initial condition $x_0 \in  \mathbb{R}^{n_x}$ is given and $w_t$ are independent and identically distributed random vectors  defined for \eqref{eq:no_control}. 
It is also assumed that a linear state feedback control $u_t = Kx_t$ has been designed such that $\|A+BK\| < 1$ for \eqref{eq:sys2} and $(A+BK,E)$ is reachable.

In this section, we design event-triggered control strategies that guarantee ultimate boundedness and positive invariance for the system \eqref{eq:sys2} with the set
\begin{align}
\mathcal{D}= \{x \in \mathbb{R}^{n_x}:\|x\|^2 - r^2 \leq 0\}. \label{eq:D_event}
\end{align}

To design trigger conditions, let us introduce the following notation:
Let the triggering time sequence $\{t_k\}_{k\in \mathbb{Z}_{\geq 0}}$, and define the state used for the control input by
\begin{align}
\hat{x}_t &= x_{t_k}, \  \forall t \in [t_k, t_{k+1}),
\end{align}
and the state error by
\begin{align}
e_t &= \hat{x}_t -x_t, \  \forall t \in [t_k, t_{k+1}).
\end{align}
Then the control law can be written as
\begin{align}
u_t = K\hat{x}_t, \  \forall t \in [t_k, t_{k+1}),
\end{align}
and the system \eqref{eq:sys2} can be written as
\begin{align}
x_{t+1} =(A+BK)x_t+BKe_t+Ew_t.
\end{align}

The rest of this section considers the event-triggering mechanism in the form of
\begin{align}
 t_{k+1} = \min\{ t > t_k : \phi(x_t, \hat{x}_t ) > \sigma\}, \ t_0 = 0, \label{eq:trigger_time}
\end{align}
where the triggering function $ \phi$ and the triggering threshold $\sigma$ are to be designed.
Note that such an event-trigger condition guarantees that $\phi(x_t, \hat{x}_t ) \leq \sigma$ for all $t \in \mathbb{Z}_{\geq 0}$.

We consider static event-triggered control strategies that use a constant error threshold $\sigma$. 

\subsection{Ultimate Boundedness}
Event-triggered control strategies that guarantee ultimate boundedness are followed from Theorem \ref{thm:r_uub}.

\begin{cor}\label{cor:et}
Suppose $r>0$ is chosen to satisfy
\begin{align}
r^2 > \frac{1}{\varepsilon}\text{Tr}(P), \label{eq:r_cond}
\end{align}
where $P\succ 0$ is the solution to the Lyapunov equation
\begin{align}
(A+BK) P(A+BK)^\top-P+E\Sigma_w E^\top = 0. \label{eq:lyap_et}
\end{align}
Then, the use of the event-triggered condition 
\begin{align}
\phi(x_t, \hat{x}_t ) = \|\hat{x}_t -x_t\| =  \|e_t\|> \sigma_1 \label{eq:trig1}
\end{align}
in \eqref{eq:trigger_time}
guarantees ultimate boundedness with \eqref{eq:D_event}
if $\sigma_1$ satisfies
\begin{align}
\sigma_1\leq \frac{1-\|A+BK\|}{\|BK\|}\left(r-\sqrt{\frac{1}{\varepsilon}\text{Tr}(P)}\right). \label{eq:sig1}
\end{align}
\end{cor}
\begin{proof}
By replacing $A$, $D$, $v_t$ in \eqref{eq:sys}  and $d$ in \eqref{eq:v} by $A+BK$, $BK$, $e_t$ in \eqref{eq:sys2} and $\sigma_1$ in \eqref{eq:trigger_time}, respectively, 
 the result follows from Theorem \ref{thm:r_uub}.
\end{proof}

A similar result can be obtained using the error threshold on the control input error:
\begin{cor}\label{cor:et2}
Suppose $r>0$ is chosen to satisfy the condition \eqref{eq:r_cond}.
The use of the static event-triggered condition
\begin{align}
\phi(x_t, \hat{x}_t ) = \|K(\hat{x}_t -x_t)\| = \|Ke_t\|> \sigma_2 \label{eq:trig2}
\end{align}
with
\begin{align}
\sigma_2\leq \frac{1-\|A+BK\|}{\|B\|}\left(r-\sqrt{\frac{1}{\varepsilon}\text{Tr}(P)}\right).\label{eq:sig2}
\end{align}
in \eqref{eq:trigger_time} guarantees ultimate boundedness with \eqref{eq:D_event}.
\end{cor}
\begin{proof}
Similarly to Corollary \ref{cor:et}, replace $A$, $D$, $v_t$ in \eqref{eq:sys}  and $d$ in \eqref{eq:v} by $A+BK$, $B$, $Ke_t$ in \eqref{eq:sys2} and $\sigma_2$ in \eqref{eq:trigger_time}, respectively, 
 the result follows from Theorem \ref{thm:r_uub}.
\end{proof}

\subsection{Positive Invariance}
Event-triggered control strategies that guarantee positive invariance are followed from Theorem \ref{thm:r_inv}.

\begin{cor}\label{cor:et3}
Suppose $r>0$ is chosen to satisfy
\begin{align}
r^2 >  \frac{1}{(1-\|A+BK\|)^2}\frac{1}{\varepsilon}\text{Tr}(\Sigma_wE^\top E).\label{eq:r_cond2}
\end{align}
Then, the use of the event-triggered condition 
\begin{align}
\phi(x_t, \hat{x}_t ) = \|\hat{x}_t -x_t\|= \|e_t\|> \sigma_3  \label{eq:trig3}
\end{align}
in \eqref{eq:trigger_time}
guarantees positive invariance with \eqref{eq:D_event}
if $\sigma_3$ satisfies
\begin{align}
\sigma_3\leq \frac{1}{\|BK\|} \left((1-\|A+BK\|) r- \sqrt{\frac{1}{\varepsilon}\text{Tr}(\Sigma_wE^\top E)}\right).\label{eq:sig3}
\end{align}
\end{cor}
\begin{proof}
Follows from Theorem \ref{thm:r_inv}.
\end{proof}

A similar result can be obtained using the error threshold on the control input error:
\begin{cor}\label{cor:et4}
Suppose $r>0$ is chosen to satisfy the condition \eqref{eq:r_cond2}.
The use of the static event-triggered function 
\begin{align}
\phi(x_t, \hat{x}_t ) = \|K(\hat{x}_t -x_t)\| = \|Ke_t\|> \sigma_4 \|x_t\|
\end{align}
with
\begin{align}
\sigma_4\leq\frac{1}{\|B\|} \left((1-\|A+BK\|) r- \sqrt{\frac{1}{\varepsilon}\text{Tr}(\Sigma_wE^\top E)}\right).\label{eq:sig4}
\end{align}
in \eqref{eq:trigger_time} guarantees positive invariance with \eqref{eq:D_event}.
\end{cor}
\begin{proof}
Follows from Theorem \ref{thm:r_inv}.
\end{proof}

\subsection{Numerical Examples}
Here we observe the performances of the event-triggered controllers in the previous subsections using numerical examples.

Consider the system \eqref{eq:sys2} with 
\begin{align}\begin{aligned}
A &= \begin{bmatrix}1.2 & 0.3\\ 0&0.5\end{bmatrix}, \ 
B = \begin{bmatrix}1 \\ 0.5 \end{bmatrix}, \
E = \begin{bmatrix}1 & 2\\ 0.5&-0.5 \end{bmatrix},\\
x_0 &=  \begin{bmatrix}2 & 3\end{bmatrix}^\top
\end{aligned}\end{align}
subject to the zero-mean Gaussian disturbance with the covariance
\begin{align}
\Sigma_w = \begin{bmatrix}0.5 & 0\\ 0&0.25 \end{bmatrix},
\end{align}
and the state feedback gain 
\begin{align}
K =  \begin{bmatrix}-0.7 & -0.2 \end{bmatrix}.
\end{align}
Choose $\varepsilon = 0.3$ to compute the worst-case CVaR. 

To see the performances of the controllers with Corollaries \ref{cor:et} and \ref{cor:et2}, choose $r = 6> \sqrt{\frac{1}{\varepsilon}\text{Tr}(P)}=2.94$ and $\sigma_1 = 1.36$, $\sigma_2 = 0.99$ which satisfy \eqref{eq:sig1} and \eqref{eq:sig2} with equalities, respectively. 
With those parameters, the event-triggered control performances and control inputs as well as those of a periodic controller (i.e., standard state feedback controller that updates the control input all the time) are shown in Figure \ref{fig:uub}.
It is observed that the number of the control input updates were reduced to 27 and 25 during 60 time-steps, respectively, while $\|x_t\|^2$ is always smaller than $15 < 6^2=36$, thus 
$ \sup_{\mathbb{P}\in \mathcal{P}_t}\mathbb{P}\text{-CVaR}_{\varepsilon}[\|x_t\|^2-r^2]\leq0$ for sufficiently large $ t$ is achieved.

\begin{figure}[tb]\centering
\includegraphics[width=.98\linewidth, viewport =0 0 400 260, clip]{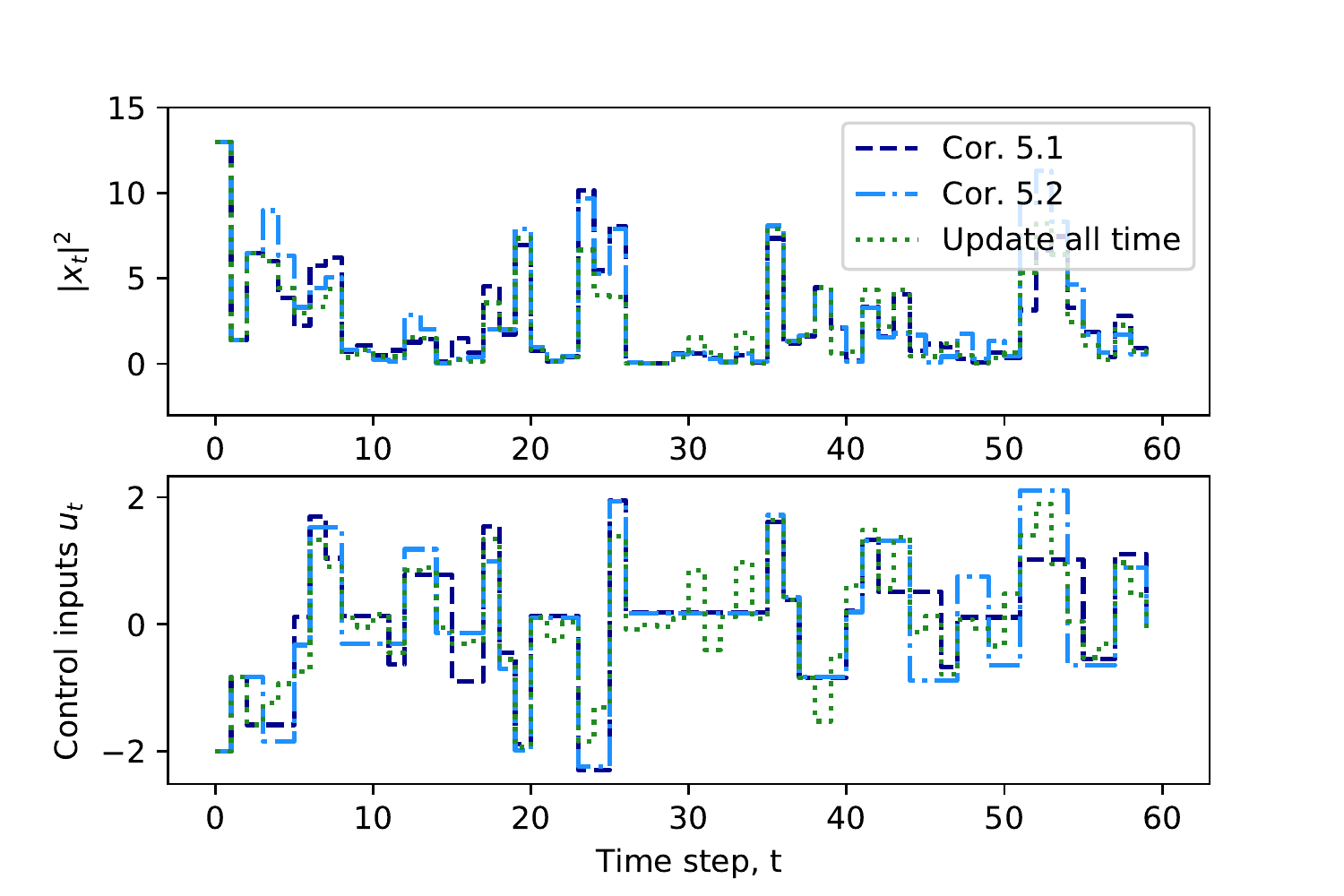}
\caption{Event-triggered controller performances and control inputs (ultimate boundedness)} 
\label{fig:uub}
\end{figure}

To see the performances of the controllers with Corollaries \ref{cor:et3} and \ref{cor:et4}, choose $r =10 > \sqrt{\frac{1}{(1-\|A+BK\|)^2}\frac{1}{\varepsilon}\text{Tr}(\Sigma_wE^\top E)}=6.54$ and $\sigma_3 = 1.52$, $\sigma_4 = 1.11$ which satisfy \eqref{eq:sig3} and \eqref{eq:sig4} with equalities, respectively. 
With those parameters, the event-triggered control performances and control inputs as well as those of a periodic controller are shown in Figure \ref{fig:pi}.
The number of control input updates were 24 and 23 during 60 time-steps, respectively.
\begin{figure}[tb]\centering
 \includegraphics[width=.98\linewidth, viewport =0 0 400 260, clip]{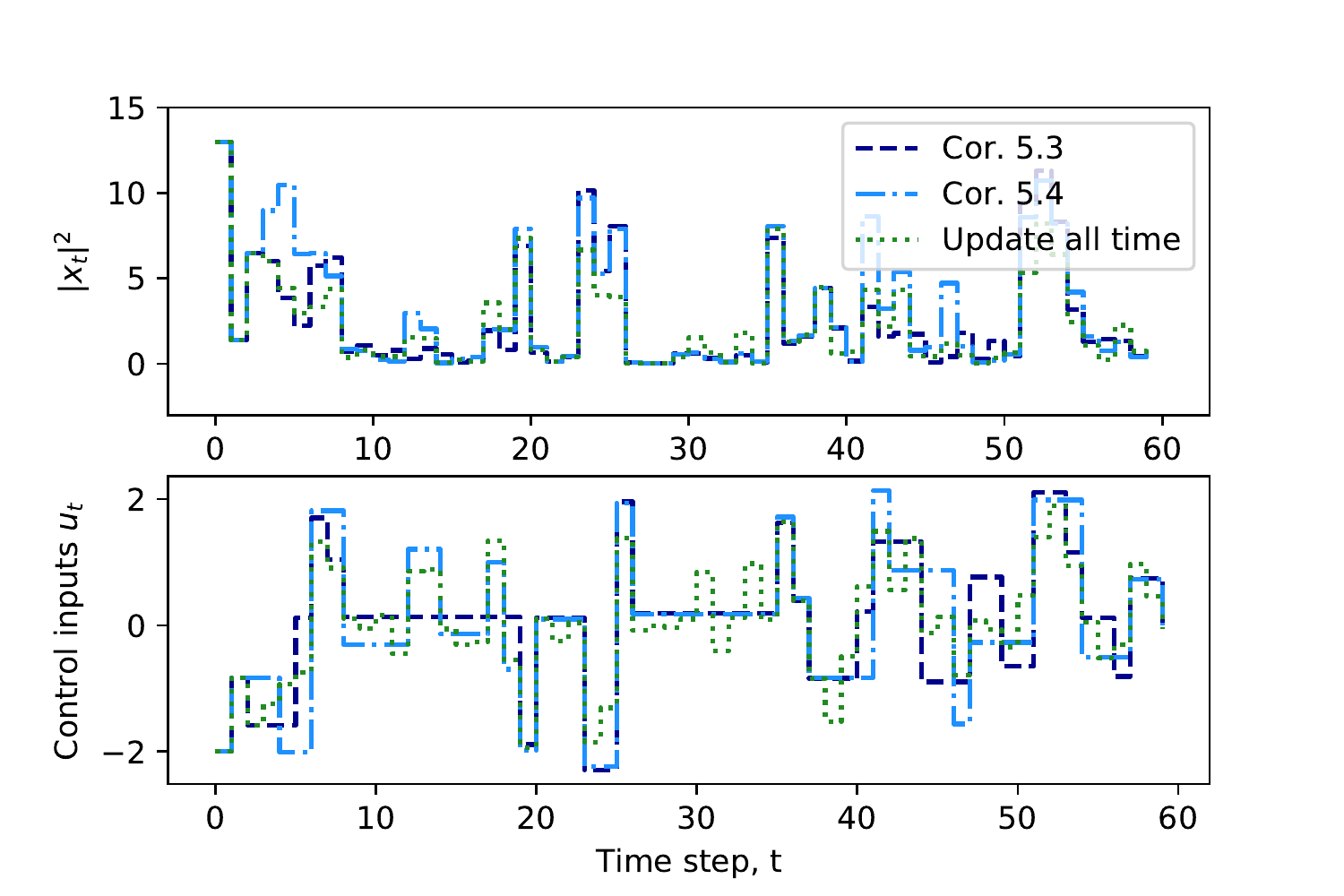}
\caption{Event-triggered controller performances and control inputs (positive invariance)} 
\label{fig:pi}
\end{figure}

In all cases, it is observed that the number of control inputs achieved a 50\% reduction while achieving the objectives. 
Those reductions are heavily dependent on the threshold $\sigma$. 
Large $r$ and $\varepsilon$, and a small $\|A+BK\|$ help to reduce the number of updates. 

\section{Conclusion}\label{sec:conc}
This paper introduced the novel concepts of stability, ultimate boundedness, and positive invariance for stochastic systems using the worst-case Conditional Value-at-Risk (CVaR) to quantify the tail behavior of the stochastic systems. These notions extend the stochastic correspondences and allow us to consider risk in the decision-making processes. 
The introduced notions are used to design event-triggered controllers and their performances were illustrated using numerical examples.


\bibliographystyle{IEEEtran}
\bibliography{IEEEabrv,myref}

\end{document}